\documentclass{article}
\usepackage{amssymb, amsthm, amsmath}
\usepackage{graphicx}
\usepackage{float}
\usepackage{subfigure}

\newtheorem{theorem}{Theorem}
\newtheorem{lemma}{Lemma}

\newtheorem{proposition}{Proposition}

\title{Non-singular Morse-Smale flows on n-manifolds with attractor-repeller dynamics}
\author{O. V. Pochinka, D. D. Shubin}
\date{HSE University}

\begin{document}
\sloppy	
	\maketitle

\begin{abstract}
In the present paper the exhaustive classification up to topological equivalence of non-singular Morse-Smale flows on $n$-manifolds $M^n$ with exactly two periodic orbits is presented. Denote by $G_2(M^n)$ the set of such flows. Let a flow $f^t:M^n\to M^n$ belongs to the set $G_2(M^n)$.  Hyperbolicity of periodic orbits of $f^t$  implies that among them one is an attracting and the other is a  repelling orbit. Due to the Poincar\'{e}–Hopf theorem, the Euler characteristic of the ambient  manifold $M^n$ is  zero. Only torus and Klein bottle can be ambient manifolds for $f^t$ in case of $n=2$. The authors established that there are exactly two classes of topological equivalence of flows in $G_2(M^2)$ if $M^2$ is the torus and three classes if $M^2$ is the Klein bottle. For all odd-dimensional manifolds the Euler characteristic is zero. However, it is known that an orientable $3$-manifold admits a flow from $G_2(M^3)$ if and only if  $M^3$  is a lens space $L_{p,q}$. In this paper it is proved that every set $G_2(L_{p,q})$ contains exactly two classes of topological equivalence of flows, except the case when $L_{p,q}$ is homeomorphic to the 3-sphere $\mathbb S^3$ or the projective space $\mathbb RP^3$, where such a class is unique. Also, it is shown that the only non-orientable $n$-manifold (for $n>2$), which admits flows from $G_2(M^n)$ is the twisted I-bundle over $(n-1)$-sphere $\mathbb S^{n-1}\tilde{\times}\mathbb S^1$. Moreover, there are exactly two classes of topological equivalence of flows in $G_2(\mathbb S^{n-1}\tilde{\times}\mathbb S^1)$. Among orientable $n$-manifolds only the product of $(n-1)$-sphere and the circle $\mathbb S^{n-1}{\times}\mathbb S^1$ can be ambient manifold for flows from $G_2(M^n)$  and $G_2(\mathbb S^{n-1}{\times}\mathbb S^1)$ splits into two topological equivalence classes.
\end{abstract}
	
\section{Introduction and statement of results}
This article will focus on  \emph{non-singular Morse-Smale flows} (abbreviated as \emph{NMS-flows}), which are Morse-Smale flows without fixed points, given on closed $n$-manifolds $M^n$, $n>1$. The authors obtained the exhaustive topological classification of NMS-flows $f^t:M^n\to M^n$ with exactly two periodic orbits. 
A general theory of hyperbolic dynamical systems (see e.g. \cite{Sm}) implies that the ambient manifold $M^n$ for such a flow $f^t$ is the union of the stable and the unstable manifolds of these orbits. It immediately implies that one of these orbits is an attracting (denote it $A$) and the other is a repelling (denote it $R$).

Let $G_2(M^n)$ be the class on NMS-flows with exactly two periodic orbits. In cases where the results are fundamentally different for orientable and non-orientable manifolds we will use notation $M^n_+, M^n_-$ for orientable and non-orientable manifolds respectively.

Recall that a periodic orbit is called \emph{twisted} if at least one from its invariant manifolds is non-orientable. Otherwise, we call the orbit \emph{untwisted}. 
Poincar\'{e}–Hopf theorem implies that the Euler characteristic of a NMS-flow ambient manifold equals to zero. Considering two-dimensional surfaces we immediately get that this constraint leaves us only the torus and the Klein bottle (actually, both admit NMS-flows). The classification of such flows follows from the classification of Morse-Smale flows on surfaces (see e.g. \cite{Peix}, \cite{OS}, \cite{MaKrPo_OmSt}). We provide an independent classification in the class $G_2(M^2)$ in section ~\ref{se:G2M2}.

\begin{theorem}\label{G2M2} $ $

\begin{enumerate}
\item The set $G_2(M^2_+)$ splits into two topological equivalence classes  of flows (see Fig.~\ref{pictorus}), both with untwisted orbits.
\item The set $G_2(M^2_+)$ splits into two topological equivalence classes  of flows (see Fig.~\ref{picKlein}), two with twisted orbits and one with untwisted orbit.
\end{enumerate}
\end{theorem}
\begin{figure}[!ht]\label{phaseG2M2_3d}
		\begin{center}
			\begin{minipage}[h]{0.49\linewidth}
\center{\includegraphics[width=1\linewidth]{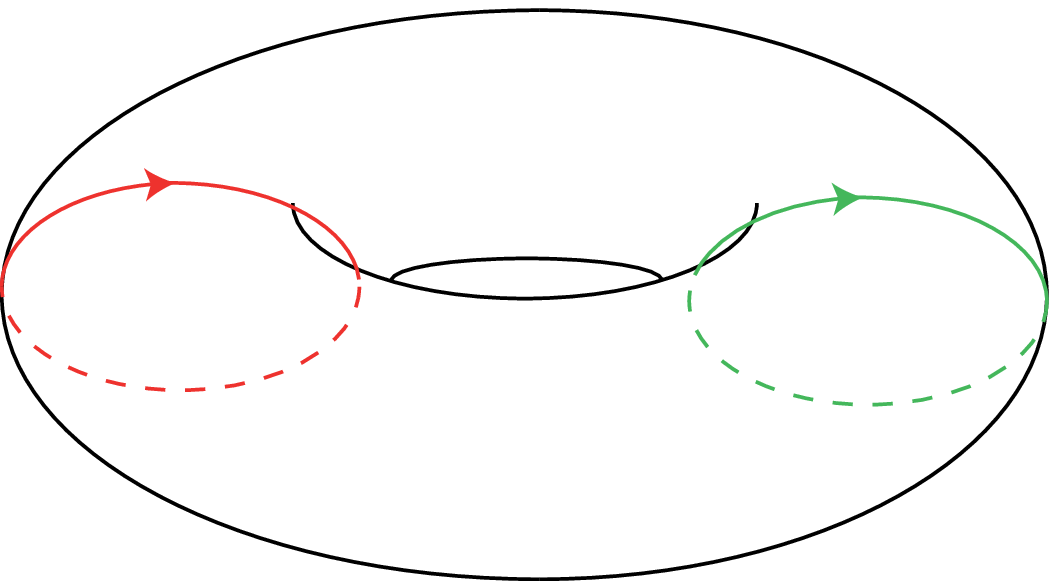} 1.}
				\end{minipage}
			\begin{minipage}[h]{0.49\linewidth}
\center{\includegraphics[width=1\linewidth]{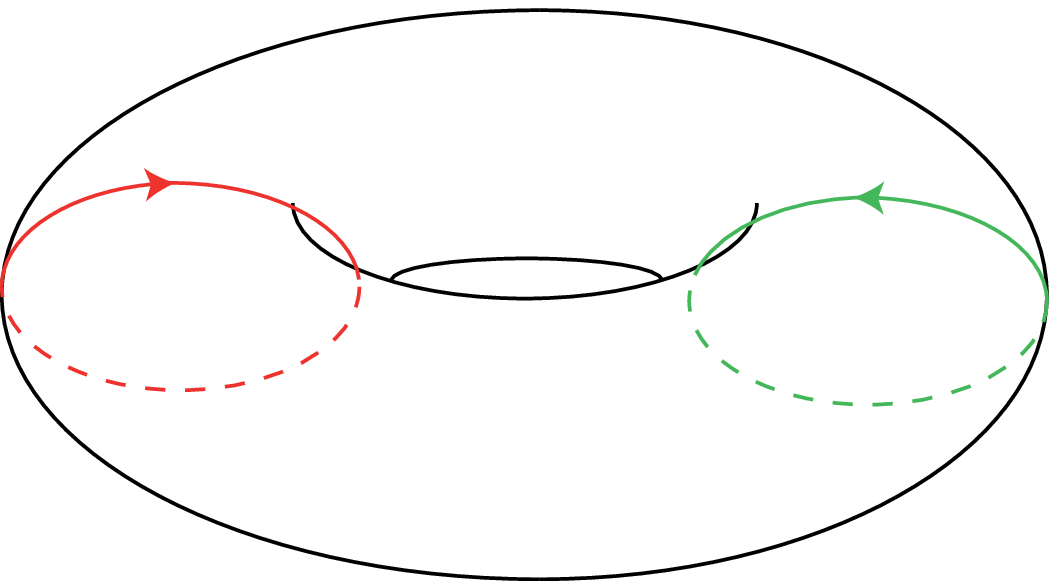} 2.}
			\end{minipage}
		\end{center}
		\caption{Phase portraits of non topologically
equivalent flows on the torus}\label{pictorus}
		\begin{minipage}[h]{0.32\linewidth}
\center{\includegraphics[width=1\linewidth]{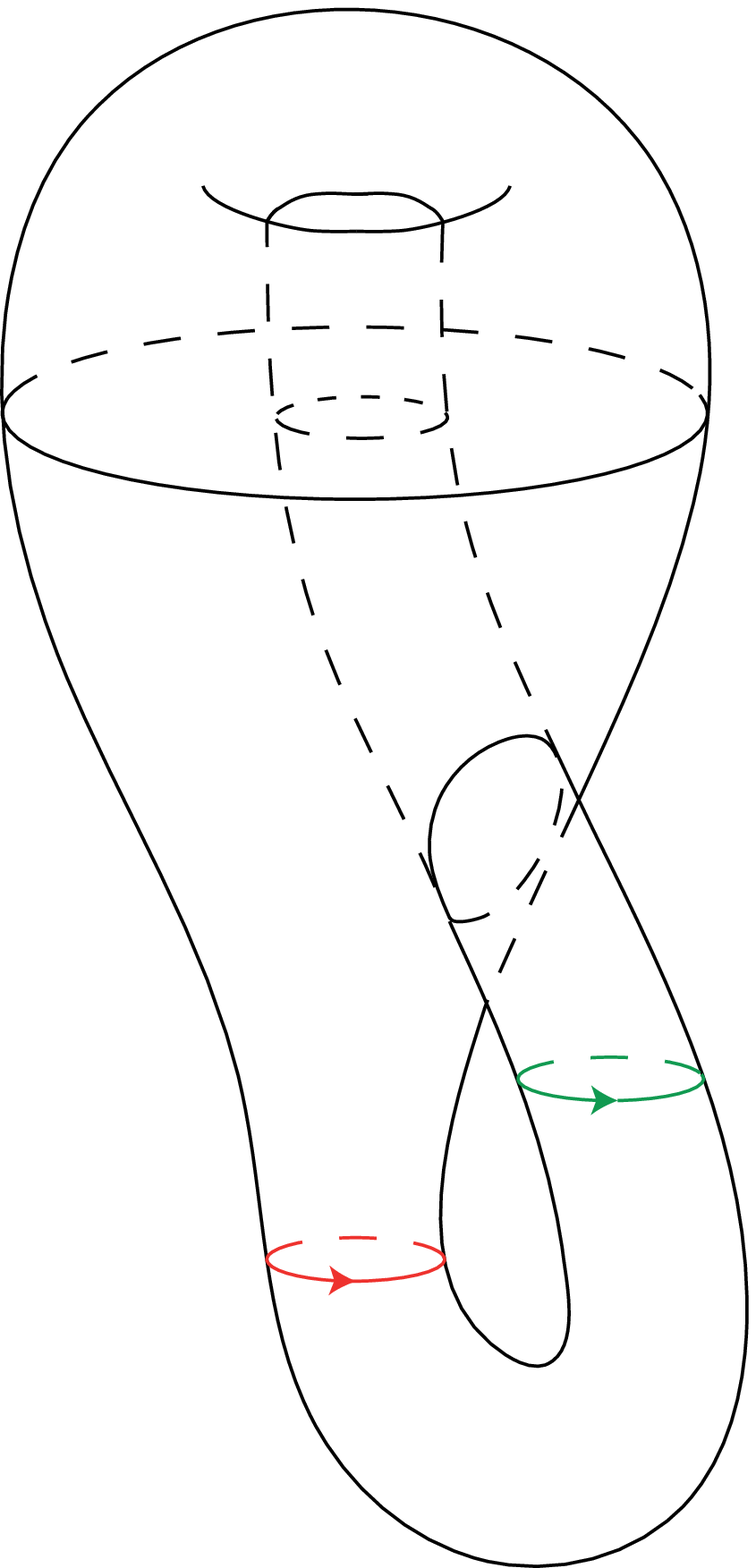} 1.}
		\end{minipage}
				\begin{minipage}[h]{0.32\linewidth}
\center{\includegraphics[width=1\linewidth]{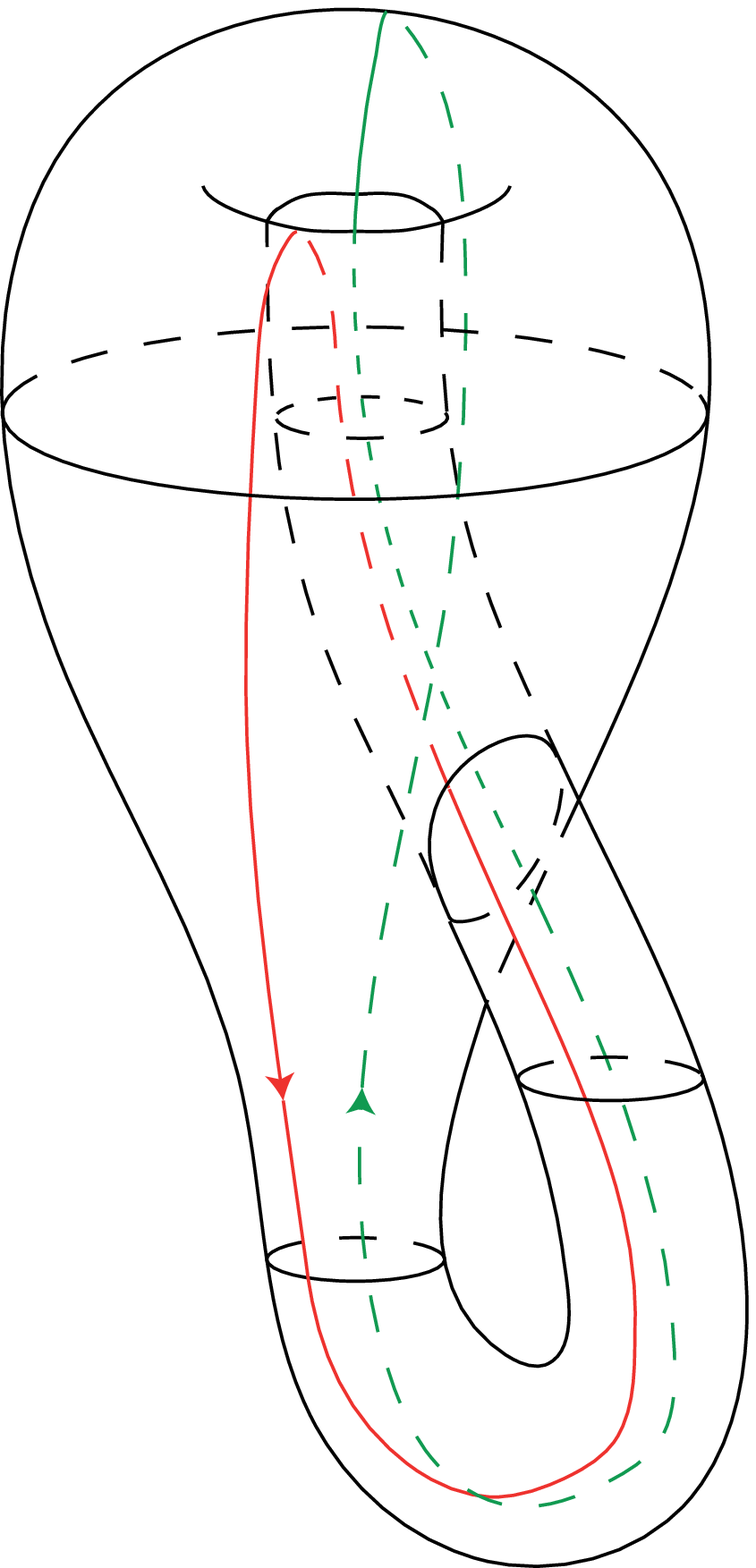} 2.}
		\end{minipage}
\begin{minipage}[h]{0.32\linewidth}
\center{\includegraphics[width=1\linewidth]{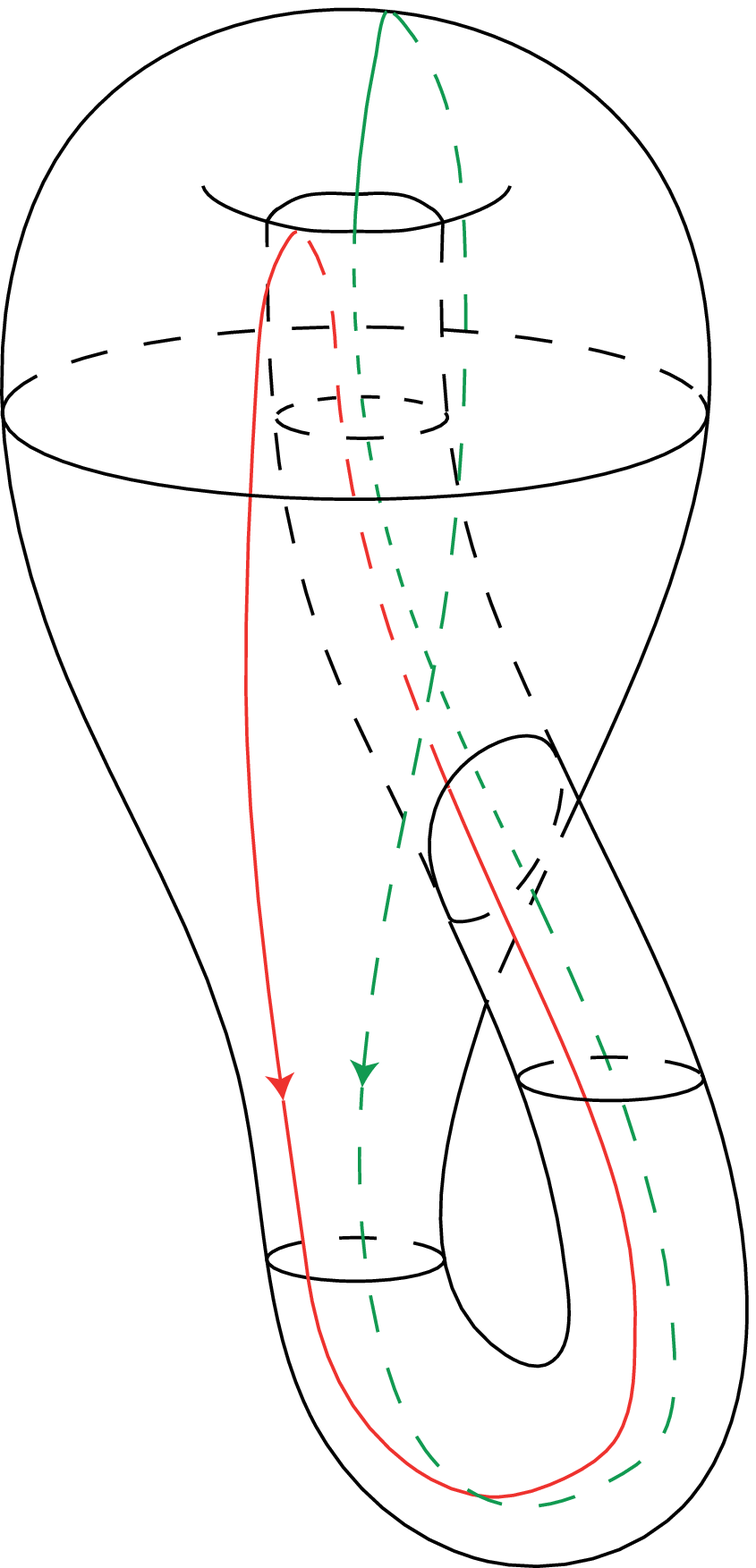} 3.}
		\end{minipage}		
		\caption{Phase portraits of pairwise non topologically
equivalent flows on the Klein bottle: 1. with untwisted orbits; 2-3. with twisted orbits}\label{picKlein}
	\end{figure}

The Euler characteristic of any  odd-dimensional manifold is zero then a priory such a manifold $M^n$ admits a flow from $G_2(M^n)$. For $n=3$ necessary and sufficient conditions for the  topological equivalence of three-dimensional NMS-flows follows from \cite{Umanski}, where larger class of Morse-Smale flows were considered. However, this classification does not allow to say anything about the admissible topology of the ambient manifolds. In the case of a small number of periodic orbits the topology of the ambient manifold and exhaustive topological classification can be established.

Recall, that a \emph{lens space} is defined as the topological space obtained by gluing two solid tori by a homeomorphism of their boundaries and is denoted as $L_{p, q},\ p,q\in \mathbb Z$, where $\langle p, q\rangle$ is the homotopy type of the meridian image under the gluing homeomorphism. Some well known $3$-manifolds are lens spaces, for example, 3-sphere $\mathbb S^3 = L_{1, 0}$, the manifold $\mathbb S^2\times \mathbb S^1 = L_{0, 1}$, the projective space $\mathbb RP^3 = L_{1, 2}$.

It follows from the proposition below that only lens spaces can be ambient manifolds for NMS-flows with a small number of the  periodic orbits.
\begin{proposition}[\cite{CamposCorderoMartinezAlfaroVindel}]\label{flows_are_on_lens}
	Let $M$ be an orientable, simple\footnote{$n$-manifold is called \emph{simple}, if it is impossible to represent it as connected sum of two $n$-manifolds each of which is not $\mathbb S^n$.}, closed 3-manifold without boundary. If $M$ admits an NMS-flow with 0 or 1 saddle periodic orbit, then $M$ is a lens space.
\end{proposition}
As every NMS-flows obligatory has at least one an attracting and at least one a repelling orbit then the complete number of orbits for a flow, satisfying to Proposition  \ref{flows_are_on_lens}, is at least two (exactly two when there is no saddle orbits at all). 

Existence and uniqueness up to topological equivalence of a flow  in the set $G_2(\mathbb S^3)$ follows from the proposition below.
\begin{proposition}[\cite{Yu}] 
	Up to topological equivalence, there exists exactly one NMS-flow $f^t:\mathbb S^3\to\mathbb S^3$ whose periodic orbits are composed of an attractor $A$ and a repeller $R$. Moreover, the periodic orbits $A\sqcup R$ form the Hopf link in $\mathbb S^3$ (see Fig. \ref{link-hoph}).
\end{proposition}
\begin{figure}[h] 	
\centerline{\includegraphics[width=0.9\textwidth]{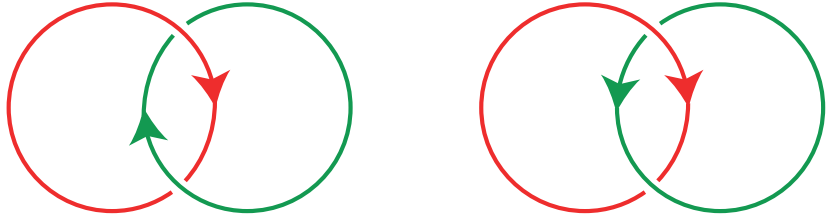}} 
	\caption{Phase portraits of equivalent flows on $3$-sphere}\label{link-hoph}
\end{figure}

In the present paper the exhaustive topological classification of class $G_2(M^3)$ is done.
\begin{theorem}\label{G2M3} $ $
	\begin{enumerate}
\item A manifold $M^3_+$ admits a flow from the set $G_2(M^3_+)$ if and only if  $M^3_+$ is a lens space. The set $G_2(M^3_+)$ splits into two topological equivalence classes of flows (see Fig.~\ref{lens_exmp_1}), except the case when $M^3_+$ is the 3-sphere $\mathbb S^3$ or the projective space $\mathbb RP^3$, where such a class is unique. The periodic orbits are untwisted in any case. 
\item The only non-orientable manifold which admits flow from the set $G_2(M^3_-)$ is the twisted I-bundle over 2-sphere $\mathbb S^2\tilde{\times}\mathbb S^1$. The set $G_2(\mathbb S^2\tilde{\times}\mathbb S^1)$ splits into two topological equivalence classes of flows, both periodic orbits of such flows are twisted.
\end{enumerate}
\end{theorem}
\begin{figure}[h!]
	\centering
	\begin{minipage}[b]{0.495\textwidth}
\center{\includegraphics[width=0.8\linewidth]{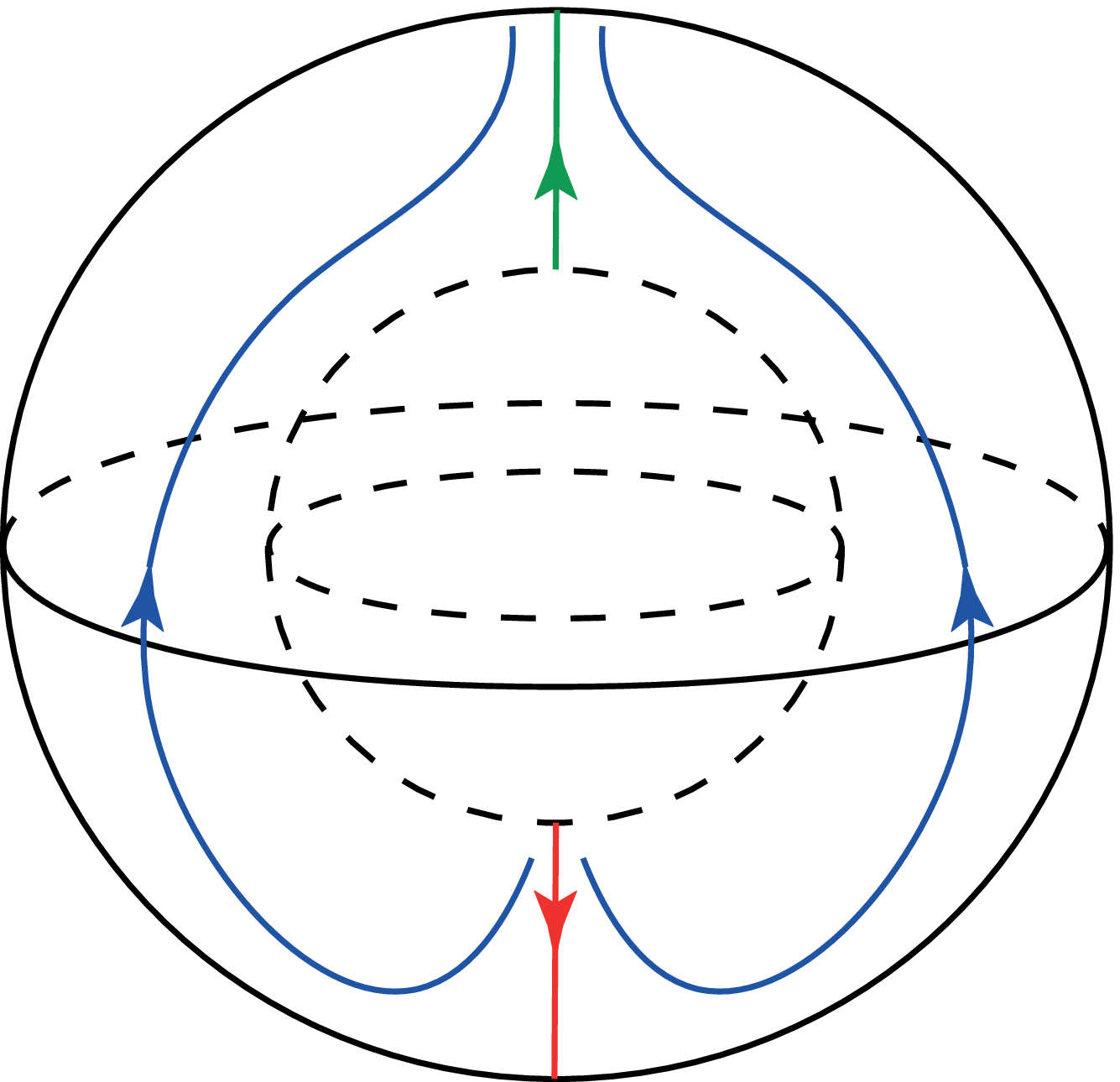}}
	\end{minipage}
	\hfill
	\begin{minipage}[b]{0.495\textwidth}
\center{\includegraphics[width=0.8\linewidth]{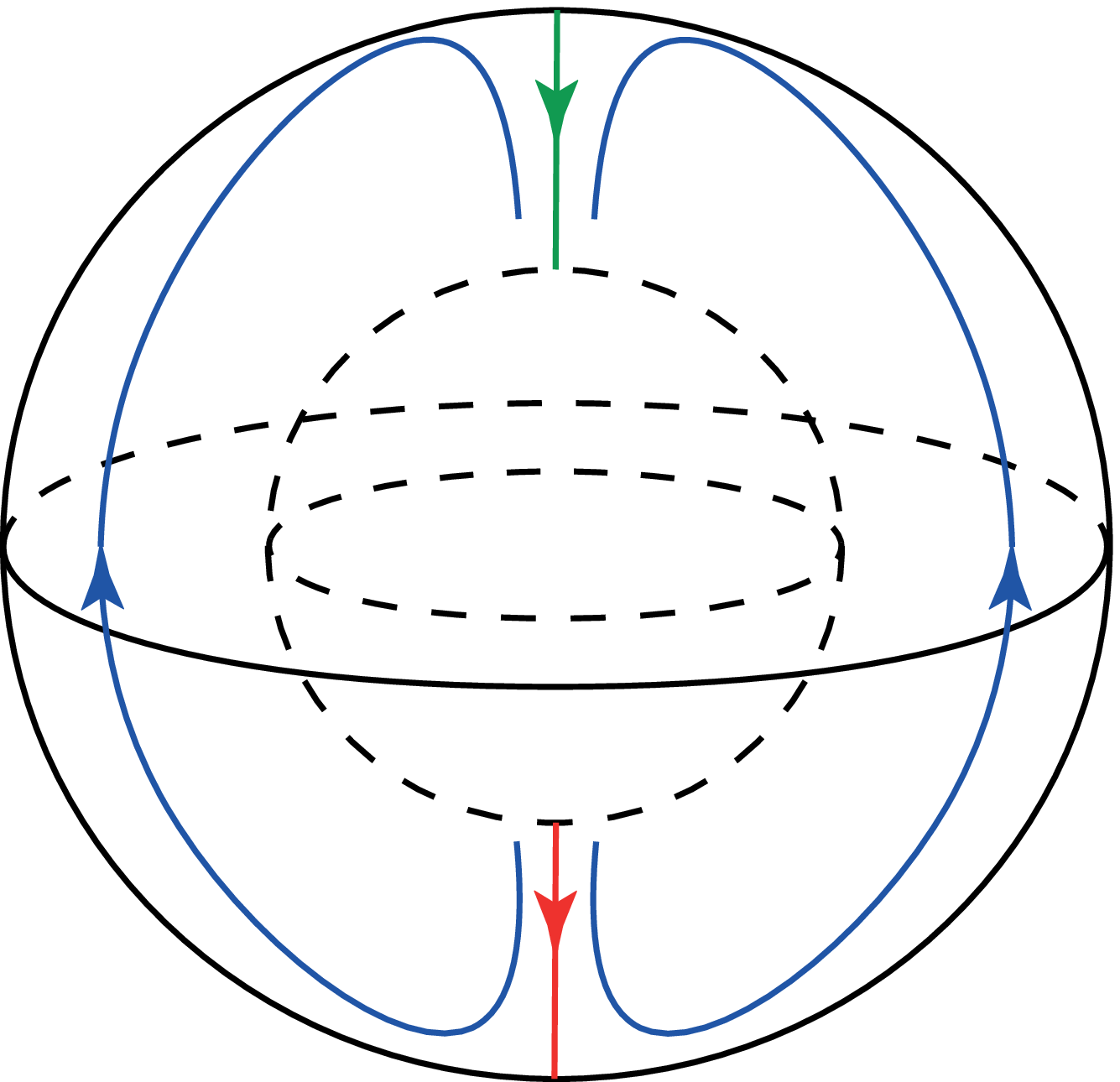}}
	\end{minipage}
	\caption{Phase portraits of non topologically
equivalent flows on $\mathbb{S}^2\times \mathbb{S}^1$}\label{lens_exmp_1}
\end{figure}

Exhaustive classification of $G_2(M^n), n>3$ follows from the  theorem below.
\begin{theorem}\label{G2Mn} $ $
\begin{enumerate}
\item A manifold $M^n_+$ admits a flow from the set $G_2(M^n_+)$ if and only if  $M^n_+$ is homeomorphic to $ \mathbb S^{n-1}\times \mathbb S^1$. The set $G_2(\mathbb S^{n-1}{\times} \mathbb S^1)$ splits into two topological equivalence classes of flows, both periodic orbits of such flows are untwisted.
\item A manifold $M^n_-$ admits a flow from the set $G_2(M^n_-)$ if and only if  $M^n_-$ is homeomorphic to $\mathbb S^{n-1}\widetilde{\times} \mathbb S^1$. The set $G_2(\mathbb S^{n-1}\widetilde{\times} \mathbb S^1)$ splits into two topological equivalence classes of flows, both periodic orbits of such flows are twisted.
\end{enumerate}
\end{theorem}

{\it Acknowledgment.} This work was supported by the Russian Science Foundation (project 21-11-00010) except for the section \ref{Crc} which is partially supported by Laboratory of Dynamical Systems and Applications NRU HSE, by Ministry of Science and Higher Education of the Russian Federation (ag. 075-15-2019-1931) and section \ref{se:G2M2} which was prepared within the framework of the Academic Fund Program at the HSE University in 2021-2022 (grant № 21-04-004).

\section{General properties of NMS-flows}\label{vsp}

In this section we provide properties of the NMS-flows which are necessary for the subsequent proofs.

Flows $f^t$ and $f'^t$ on a  manifold $M^n$ are said to be \emph{topologically equivalent} if there is a homeomorphism $h\colon M^n\to M^n$ which sends orbits of $f^t$ into orbits of $f'^t$ and preserves the orientation on the orbits.

To describe the behaviour of a flow $f^t\colon M^n\to M^n$ in a neighbourhood of an attracting or repelling hyperbolic orbit we use the following notion of a  \emph{suspension}.

Define a diffeomorphism $a_{\pm}:\mathbb R^{n-1}\to\mathbb R^{n-1}$ by  
$$a_ {\pm} (x_1,x_2, ..., x_{n-1}) = (\pm 2x_1,2x_2, ..., 2x_{n-1}).$$
Let $g_\pm\colon \mathbb{R}^n \to \mathbb{R}^n$ be a diffeomorphism defined by 
$$g_\pm(x, r)=(a_{\pm}(x), r - 1).$$
Let $\Pi_{\pm}=\mathbb R^{n}/{\langle g_{\pm}\rangle}$ and denote the natural projection by $v_{\pm}\colon \mathbb R^n\to\Pi_{\pm}$.
Define a flow $\bar b^t$ on $\mathbb R^n$ by the system of the following differential equations:
$$
\left\{
\begin{aligned}
	\dot x_1=0, \\
	\dots, \\
	\dot x_{n-1}=0, \\
	\dot x_n=1.\\
\end{aligned}
\right.$$

The natural projection $v_{\pm}$ induces a flow $b^t_{\pm}=v_{\pm}\bar b^t v_{\pm}^{-1}:\Pi_{\pm}\to\Pi_{\pm}$ which is called \emph{suspension}.

\begin{proposition}[\cite{Irwin}]\label{prop:irwin-eqv}
Every hyperbolic repelling orbit $R$ of a flow $f^t\colon M^n\to M^n$ possesses the unstable manifold $W^u_{R}=\{x\in S\mid f^t(x)\to R\,\, \text{as}\,\, t\to-\infty\}$ with the following properties:
\begin{enumerate}
\item there is a value $\delta_R\in\{-,+\}$ and a homeomorphism $h_R\colon \Pi_{\delta_R}\to W^u_R$ which provides the topological equivalence of the flows $b^t_{\delta_R}$ and $f^t|_{W^u_R}$. The orbit $R$ is twisted, if $f^t|_{W^u_R}$ is equivalent to $b^t_-$ and is untwisted otherwise.
\item $W^u_R$ is diffeomoprhic to $\mathbb R^{n-1}\times \mathbb S^1$ if $R$ is untwisted and is diffeomorphic to $\mathbb R^{n-1} \widetilde{\times} \mathbb S^1$ if $R$ is twisted.
	\end{enumerate}
\end{proposition}
\begin{figure}[h!]
	\centering
	\begin{minipage}[b]{0.495\textwidth}
\center{\includegraphics[width=0.8\linewidth]{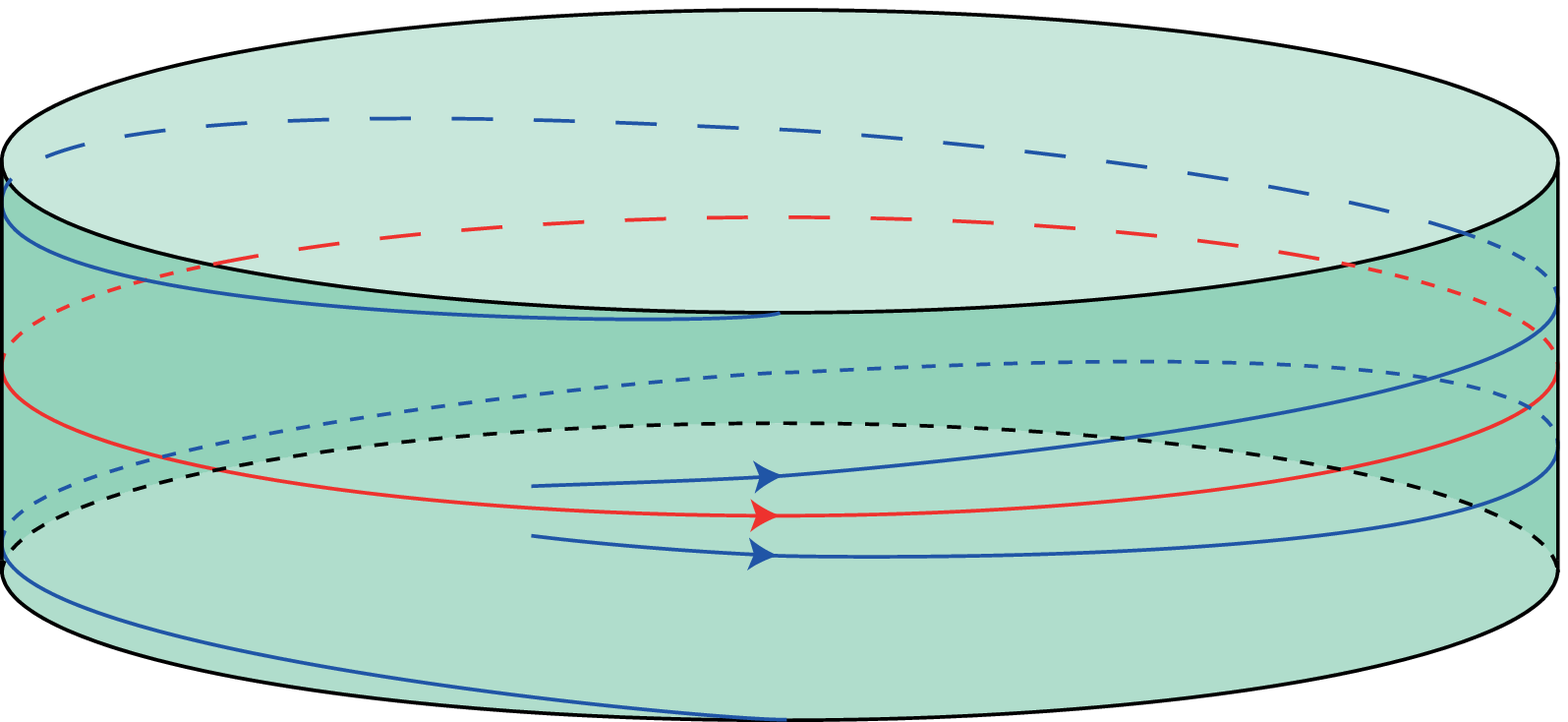}}
		\\
		Untwisted orbit for $n=2$
	\end{minipage}
	\hfill
	\begin{minipage}[b]{0.495\textwidth}
\center{\includegraphics[width=0.8\linewidth]{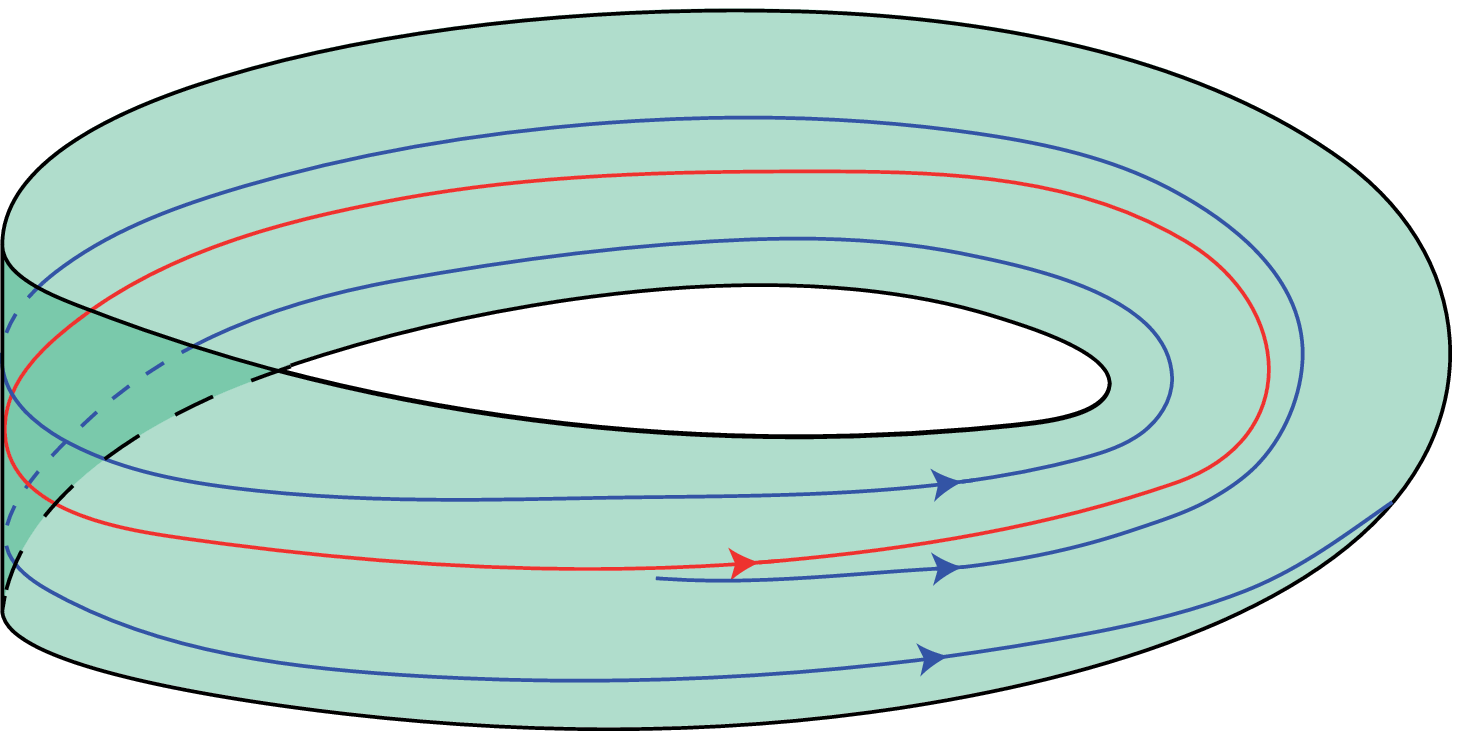}}
\\ 		Twisted orbit for $n=2$ 
	\end{minipage}
\end{figure}

A similar statement takes place  for the stable manifold $W^s_A=\{x\in S\mid f^t(x)\to A\,\, \text{as}\,\, t\to+\infty\}$ of the hyperbolic attracting orbit $A$ which states that a flow $b^{-t}_{\delta_A},\,\delta_A\in\{-,+\}$  is topologically equivalent to the flow $f^t_{W^s_A}$ by means a homeomorphism $h\colon W^s_A\to \Pi_{\delta_A}$.
\begin{figure}[h]
\centerline{\includegraphics[width=0.8\textwidth]{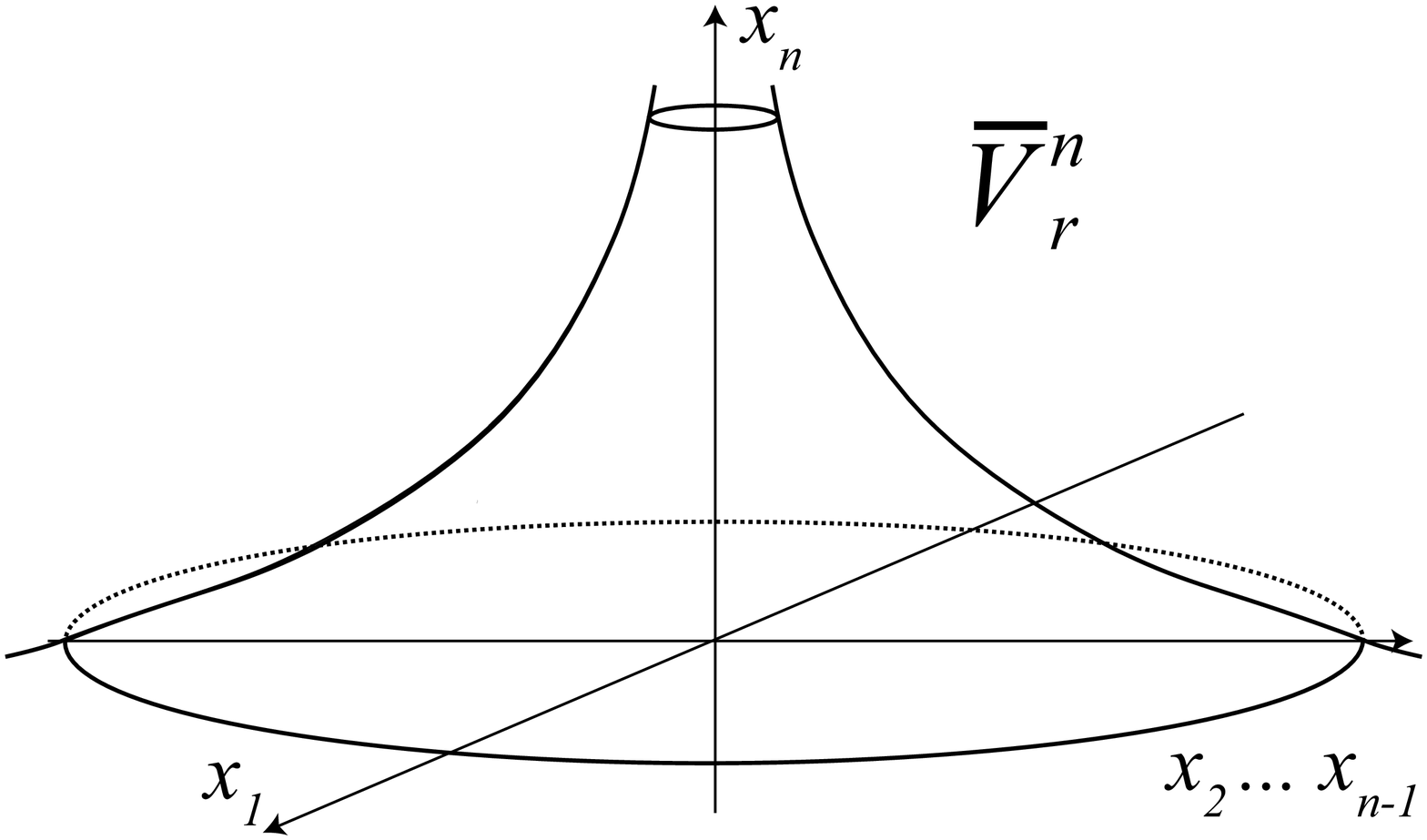}}
	\caption{The set $\bar{V}^n_r$}\label{Vnr0}
\end{figure}

For $r>0$ let (see Fig.~\ref{Vnr0}) $$\bar{V}^n_r= \left\{(x_1, \dots, x_n)\in \mathbb{R}^n \Big|\ x_1^2 + \dots + x_{n-1}^2 \leqslant r2^{-x_n} \right\},\quad \mathbb V^n_{\pm r}=v_\pm(\bar V^n_r).$$ 
By the construction, the quotient space $\mathbb V^n_{+r}$ is homeomorphic to the \emph{generalized solid torus} $\mathbb D^{n-1}\times \mathbb S^1$ and the quotient $\mathbb V^n_{-r}$ is homeomorphic to the  \emph{generalized solid Klein bottle} $\mathbb D^{n-1} \widetilde{\times} \mathbb S^1$. 

Let $\bar{V}^n=\bar{V}^n_1,\,
\mathbb V^n_\pm=v_\pm(\bar V^n),\,\mathbb L^n_\pm=v_\pm(Ox_n)$, where $Ox_i$ is the coordinate axis. Consider a homeomorphism $j\colon\partial \mathbb V^n_\pm\to \partial \mathbb V^n_\pm$, two copies 
$\mathbb V^n_\pm\times\mathbb Z_2,\,\mathbb Z_2=\{0,1\}$ of the manifold $\mathbb V^n_\pm$ and a  homeomorphism $J\colon\partial \mathbb V^n_\pm\times\{0\}\to \partial \mathbb V^n_\pm\times\{1\}$, defined by  $J(s,0)=(j(s),1)$.
Let $$M^n_{j}=\mathbb V^n_\pm\times\{0\}\cup_J\mathbb V^n_\pm\times\{1\}.$$
Denote the natural projection by $p_{j}\colon\mathbb V^n_\pm\times\mathbb Z_2\to M^n_{j}$ and let $f^t_{j}\colon M^n_{j}\to M^n_{j}$ be a topological flow defined by 
$$f^t_{j}(x)=\left\{
\begin{aligned}
	&p_j b^{t}_\pm(p_{j}|_{\mathbb V^n_\pm\times\{0\}})^{-1}(x),\,x\in  p_{j}|_{\mathbb V^n_\pm\times\{0\}},\,t\leqslant 0\\
	&p_j b^{-t}_\pm(p_{j}|_{\mathbb V^n_\pm\times\{1\}})^{-1}(x),\,x\in  p_{j}|_{\mathbb V^n_\pm\times\{1\}},\,t\geqslant 0.\\
\end{aligned}
\right.$$
\begin{figure}[h!]
\centerline{\includegraphics[width=0.8\textwidth]{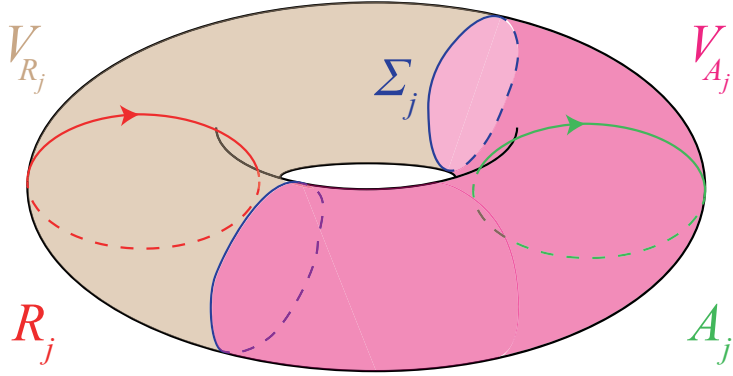}}
	\caption{Phase portrait of a model flow on the torus}\label{model_defi}
\end{figure} 

We call $f^t_j:M^n_j\to M^n_j$  \emph{$n$-dimensional model flows}. For any model flow let (see Fig.~\ref{model_defi})
$$R_{j}=p_{j}({\mathbb L^n_\pm\times\{0\}}),\,A_{j}=p_{j}({\mathbb L^n_\pm\times\{1\}}),$$
$$V_{R_{j}}=p_{j}({\mathbb V^n_\pm\times\{0\}}),\,V_{A_{j}}=p_{j}({\mathbb V^n_\pm\times\{1\}}),$$ 
$$\Sigma_{j}=p_{j}({\partial\mathbb V^n_\pm\times\{0\}})=p_{j}({\partial\mathbb V^n_\pm\times\{1\}}).$$

\begin{lemma}\label{redu}
Every flow $f^t\in G_2(M^n)$ is topologically equivalent to some model flow $f^t\colon M^n_j\to M^n_j$.
\end{lemma}
\begin{proof} Let $f^t\in G_2(M^n)$ and $A,R$ be its the attracting and the repelling hyperbolic orbits respectively. Due to Proposition~\ref{prop:irwin-eqv} there is a homeomorphism $h_R\colon \Pi_\pm\to W^u_R$ which provides the topological equivalence of the flows $b^t_\pm$  and $f^t|_{W^u_R}$. Also, there is a homeomorphism $h_A\colon \Pi_\pm\to W^s_A$ which provides the topological equivalence of the flows $b^{-t}_\pm$ and $f^t|_{W^s_A}$. Let $V_A=h_A(\mathbb V^n_\pm)$ and $H_A=h_A|_{\mathbb V^n_\pm}$. We can choose $r>0$ such that a  neighbourhood $V'_R=h_R(\mathbb V^n_{\pm r})$ of $R$ is disjoint with $V_A$. Since the non-wandering set of $f^t$  consists of $A$ and $R$ then (see e.g. \cite{Sm})
	$$M^n=W^u_R\cup A=W^s_A\cup R$$ and consequently the set $M^n\setminus int(V_A \cup V'_R)$ consists of segments of wandering trajectories of the flow $f^t$, which have their boundary points on different connected components of the boundary $\partial M^n\setminus int(V_A \cup V'_R)$ (see Fig.~\ref{lemm1_proof}).
\begin{figure}[h]
\centerline{\includegraphics[width=0.8\textwidth]{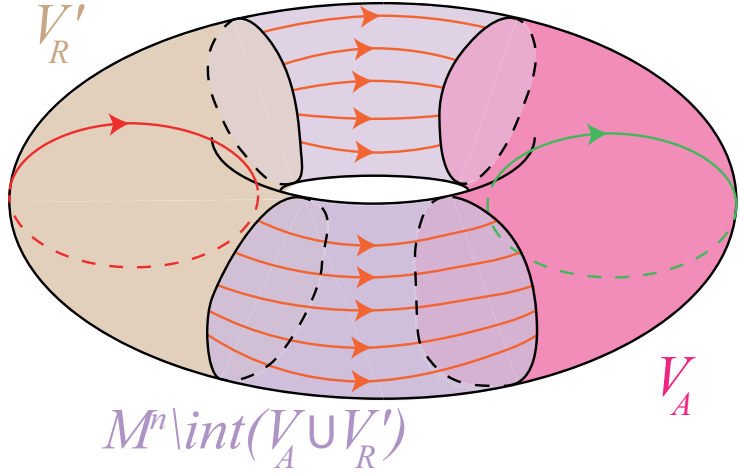}}
\caption{The neighbourhoods $V_A$ and $V'_R$ on $M^n$}\label{lemm1_proof}
\end{figure}
Let $V_R=M^n\setminus int\, V_A$. 
Then the homeomorphism $h_R|_{\mathbb V^n_{\pm r}}$ can be extended to the homeomorphism $H_R\colon\mathbb V^n_\pm\to V_R$ which provides the topological equivalence of the flows $b^t_\pm$ and $f^t$. 
Define a homeomorphism $j\colon\partial\mathbb V^n_\pm\to\partial\mathbb V^n_\pm$ by  $$j=H_A^{-1}H_R|_{\partial\mathbb V^n_\pm}.$$ 
Define a homeomorphism $H\colon M^n_{j}\to M^n$ by  $$H(x) = \begin{cases}
	H_A p_{j}^{-1}(x), & x\in V_{A_{{j}}}\\
	H_R p_{j}^{-1}(x), & x\in V_{R_{{j}}}	 
\end{cases}.$$ It is  directly verified that $H$ provides the  topological equivalence of the flows $f^t_j$ and $f^t$.
\end{proof}

Thus, the classification of the flows with the attractor-repeller dynamics can be reduced to the classification of the model flows.  Using methods of the previous proof it is easy to show that it is sufficient to consider some special class of homeomorphisms providing the topological  equivalence of the model flows.

\begin{lemma}\label{sig}
If the model flows $f^t_j\colon M^n_{j}\to M^n_{j}$ and $f^t_{j'}\colon M^n_{j'}\to M^n_{j'}$ are topologically equivalent, then there is a providing the topological  equivalence of these  flows homeomorphism $H\colon M^n_{j}\to M^n_{j'}$ such that $H(\Sigma_{j})=\Sigma_{j'}$.
\end{lemma}

\section{A criteria for the model flows topological equivalence}\label{Crc}

In this section we give a criteria of the model flows topological equivalence, from which the complete description of equivalence classes in $G_2(M^n)$ follows. 

Denote  by $\bar \alpha$ the connected component of the set $\partial \bar V^n\cap Ox_2x_n$ containing the point $(0, 1, \dots, 0)$. We consider $\bar \alpha$ as the curve oriented in the direction of the increasing $x_n$ coordinate. Let $\bar \beta=\partial \bar V^n\cap Ox_1\dots x_{n-1}$ (see Fig.~\ref{V_a_b_n}) and $$\alpha_\pm = v_\pm(\bar \alpha), \beta_\pm = v_\pm(\bar \beta).$$
\begin{figure}[h]
\centerline{\includegraphics[width=0.8\textwidth]{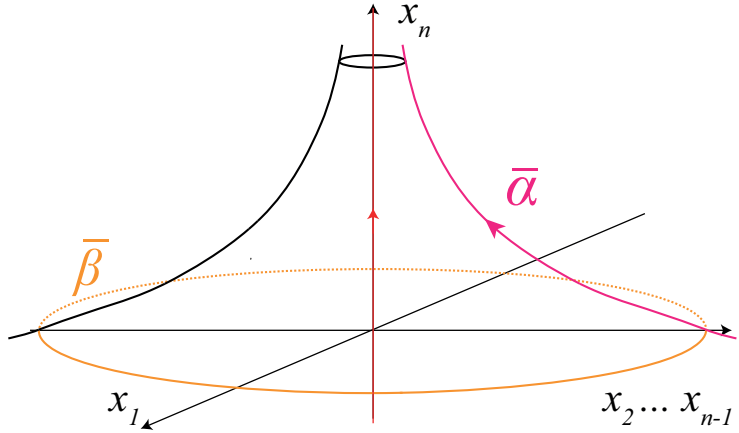}}
	\caption{$\bar \alpha$ and $\bar \beta$ on $\bar V^n$}\label{V_a_b_n}
\end{figure}

Denote by $i\colon \partial \mathbb V^n_\pm \to \mathbb V^n_\pm$ the inclusion map and by $i_*\colon \pi_1(\partial \mathbb V^n_\pm)\to \pi_1(\mathbb V^n_\pm)$ the induce isomorphism. Since the group ${\langle g_{\pm}\rangle}$ is isomorphic to $\mathbb Z$ and acts freely and discontinuously on the simply connected space $\bar V^n$, the fundamental group $\pi_1(\mathbb V^n_\pm)$ is also isomorphic to the group $\mathbb Z$ (see e.g.~\cite{Kosn}) and $\alpha_\pm$ is its generator\footnote{The space $\partial \mathbb V^2_+$ consists of two connected components. In this case $i_*$ is a map composed of the induced isomorphisms for each connected component.}.
\begin{theorem}[Criteria for topological equivalence]
Two model flows $f^t_j\colon M^n_j\to M^n_j, f^t_{j'}\colon M^n_{j'}\to M^n_{j'}$ are topologically equivalent if and only if  there is a homeomorphism $h_0\colon  \partial \mathbb V^n_\pm \to \partial \mathbb V^n_\pm$ such that
	\begin{equation}\label{i_0}
		i_*h_{0*} = i_*
	\end{equation}
	and the homeomorphism $h_1 = j' h_0j^{-1}$ possesses the property
	\begin{equation}\label{+}
		i_*h_{1*} = i_*
	\end{equation}\label{criteria}
\end{theorem}
\begin{proof} 
\textit{Necessity.} Let the flows $f^t_j$ and $f^t_{j'}$ be topologically equivalent by means of a homeomorphism $H\colon M^n_j\to M^n_{j'}$. By the Lemma~\ref{sig} without loss of generality we assume that $H(\Sigma_{j})=\Sigma_{j'}$.
Define a homeomorphism $H_k\colon \mathbb V^n_\pm \to \mathbb V^n_\pm, k\in \mathbb Z_2$ by  $(H_k(s), k) = p^{-1}_{j'} H p_j\big|_{\mathbb V^n_\pm \times \{k\}}(s, k)\colon \mathbb V^n_\pm\times \{k\} \to \mathbb V^n_\pm\times \{k\}$. Let $h_k = H_ k|_{\partial\mathbb V^n_\pm}$.
	
Notice that the curves $\mathbb L^n_\pm$ and $\alpha_\pm$ are homotopic in $\mathbb V^n_\pm$ as they bound a two-dimensional annulus $v_\pm(\bar V^n\cap Ox_2x_n)$ in $\mathbb V^n_\pm$.
As $H$ provides the  topological equivalence of the flows $f^t_j$ and $f^t_{j'}$ then $H_0(R_j) = R_{j'}$, that implies $H_{0*}([R_j]) = [R_{j'}]$. Considering the fact that $\pi_1(\mathbb V^n_\pm) \cong \langle \alpha_\pm \rangle$ we can deduce that $H_{0*} = id$ which implies equality (\ref{i_0}) $i_*h_{0*} = i_*$.
	
It follows from the definition of the model flow that	$h_1 = j' h_0j^{-1}$. The equality  (\ref{+}) $i_*h_{1*} = i_*$ for the map $h_1$ can be proved as above.
	
\textit{Sufficiency.} Assume that  there is a homeomorphism $h_0\colon  \partial \mathbb V^n_\pm \to \partial \mathbb V^n_\pm$ such that $i_*h_{0*} = i_*$ and the homeomorphism $h_1 = j' h_0j^{-1}$ satisfies $i_*h_{1*} = i_*$. Since $i_* h_{0*} = i_*$ and due to \cite{begin}[Proposition 10.2.26], homeomorphism $h_0$ admits a lift $\bar h_0\colon \partial V^n\to \partial \bar V^n$, which commutes with $g_\pm$.
	Let $\bar \beta' = \bar h_0(\bar \beta)$. Choose a positive integer $n_0$ such that $\bar \beta'\subset\{(x_1,\dots x_{n-1}, x_n)\in\mathbb R^n: 0<x_n<n_0\}$ (see Fig.~\ref{meri_n}). Let us extend the homeomorphism $\bar h_0$ to a homeomorphism $\bar H_0\colon \bar V^n\to \bar V^n$ commuting with $g_\pm:\mathbb R^n\to\mathbb R^n$ and providing the topological  equivalence of the flow $\bar b^t_\pm$ with itself. 
	
Recall that $g_\pm(x, r)=(a_{\pm}(x), r - 1).$ Let $y=(x_1,\dots,x_{n-1})$ and define a map $p_0\colon\mathbb R^n\to Oy$ by  $$p_0(y, x_n) = y.$$  By the construction $p_0|_{\partial{\bar V}^n_r}:\partial{\bar V}^n_r\to Oy\setminus O$ (hyperplane $x_n = 0$ without initial) is a diffeomorphism for any $r>0$.
	Define the homeomorphism $\bar w\colon Oy\setminus O\to Oy\setminus O$ by  $$\bar w = p_0\bar h_0({p_0}|_{\partial \bar V^n_1})^{-1}.$$ 
	Since $i_*h_{0*} = i_*$ the homeomorphism $\bar w$ can be continuously extended to the point $O$ by  $\bar w(O)=O$. Assume that $y'=\bar w(y),\,y\in Oy$, $\bar \beta'_0 = p_0(\bar \beta')$ and $\bar \beta_0=p_0(g_\pm^{-n_0}(\bar\beta))$.
	Let $\bar B,\,\bar B_0,\,\bar B'_0$ denote the disks in $Oy$ which are bounded by the spheres $\bar \beta,\,\bar \beta_0,\,\bar \beta'_0$ respectively.
	Thus, $\bar w(\bar B)=\bar B'_0$.
	
By virtue of the annulus conjecture and the fact that spheres $\bar\beta_0$ and $\bar\beta'_0$ are disjoint and are cylindrically embedded they bound an annulus $K_0\subset Oy$ (see Fig.~\cite{begin}). Let $\tau\colon\mathbb S^{n-2}\times[0,1]\to K_0$ be a homeomorphism such that $\tau(\mathbb S^{n-2}\times\{0\})=\bar \beta_0$ and $\tau(\mathbb S^{n-2}\times\{1\})=\bar \beta'_0$. For $t\in[0,1]$ put $c_t=\tau(\mathbb S^{n-2}\times\{t\})$, $r_t=t+2^{-n_0}(1-t)$ and $C_t=(p_0|_{\partial{\bar V}^3_{r_t}})^{-1}(c_t)$. Let us define a disk $\bar B'$ by  $$\bar B'=\bigcup\limits_{t\in[0,1]}C_t\cup\bar B_0.$$
	\begin{figure}
		\center\includegraphics[width=0.9\textwidth]{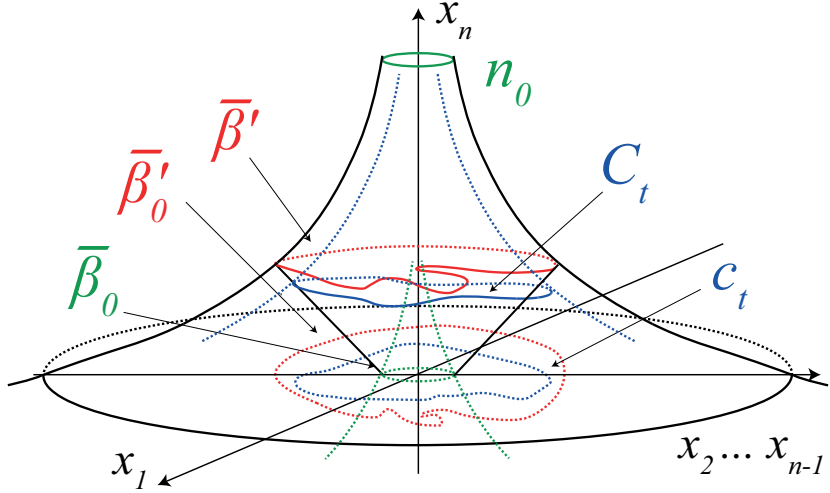}
		\caption{Construction of the disk $B'$}\label{meri_n}
	\end{figure}
	
	Since $i_*h_{0*} = i_*$, the homeomorphism $\bar h_0$ sends the part $\bar\Gamma$ of the cylinder $\partial \bar V^n$ lying between spheres $\bar\beta$, $g^{-1}_\pm(\bar\beta)$ to the part $\bar\Gamma'$ of cylinder $\partial \bar V^n$ lying between the spheres $\bar\beta'$, $g^{-1}_\pm(\bar\beta')$. Denote by $\bar W$ ($\bar W'$) a compact subset of $\bar V^n$ bounded by $\bar B$, $g^{-1}_\pm(\bar B)$ and $\bar\Gamma$ ($\bar B'$, $g^{-1}_\pm(\bar B')$ and $\bar\Gamma'$). 
	
For every $y_0\in Oy$ let $L_{y_0}=\{(y,x_n)\in\mathbb R^n:y=y_0\}$. For $(y_0,0)\in\bar B$ let $I_{y_0}=L_{y_0}\cap\bar W$ and $I'_{y'_0}=L_{y'_0}\cap\bar W'$. 
Denote boundary points of segments $I_{y_0}$ and $I'_{y'_0}$ by $A_{y_0},B_{y_0}$ and $A'_{y'_0},B'_{y'_0}$ where $A_{y_0}=(y_0,0),\,B_{y_0}=(y_0,b_{y_0})$ and $A'_{y'_0}=(y'_0,a'_{y'_0}),\,B'_{y'_0}=(y'_0,b'_{y'_0}),\,a'_{y'_0}\leqslant b'_{y'_0}$. 
	Define a homeomorphism $\bar h_{y_0}:I_{y_0}\to I'_{y'_0}$ by  
	$$\bar h_{y_0}(y_0,x_n) = \left(y'_0, x_n\dfrac{b'_{y'_0} - a'_{y'_0}}{b_{y_0}} + a'_{y'_0}\right).$$
By virtue of the fact that $\bar W=\bigcup\limits_{y_0\in\bar B}I_{y_0}$ we get a homeomorphism $h_{\bar W}:\bar W\to\bar W'$, composed of $\bar h_{y_0}$, which provides the topological equivalence  of the flow $\bar b^t_\pm|_{\bar W}$ with  $b^t_\pm|_{\bar W'}$.
	Extend $h_{\bar W}$ to $\bar V^n$ by  
	$$\bar H_0(x_1,\dots x_{n-1}, x_n) = g^{-[x_n]}(h_{\bar W}(g^{[x_n]}(x_1,\dots x_{n-1}, x_n))),$$
	where $[x]$ denotes the integer part of the number $x$. 
By the construction $\bar H_0 g_\pm = g_\pm \bar H_0$. By virtue of \cite{begin} this fact implies, that $H_0 = v^{-1}_\pm \bar H_0 v_\pm$ is a homeomorphism and the following equality holds $H_0 b^t = b^t H_0.$ 
	
By the same way the homeomorphism $h_1$ can be extended to a homeomorphism $H_1\colon \mathbb{V}^n_\pm\to \mathbb{V}^n_\pm$ commuting with $g_\pm$ and providing the topological  equivalence of the flow $\bar b^{-t}_\pm$ with itself. Thus, the requirement homeomorphism 
	$H\colon M^n_j\to M^n_{j'}$ is defined by 
	$$H(x) = \begin{cases}
		p_{j'} H_0 p_{j}^{-1}(x), & x\in p_j(\mathbb{V}^n_\pm\times\{ 0\}) \\
		p_{j'} H_1 p_{j}^{-1}(x), & x\in p_j(\mathbb{V}^n_\pm\times\{ 1\})
	\end{cases}.$$
\end{proof}

\section{Classification of surface  model flows}\label{se:G2M2}

In this section we prove  Theorem~\ref{G2M2}. 

Let $f^t_j\colon M^2_{j}\to M^2_{j}$ be a two-dimensional model flow. Then the ambient surface $M^2_{j}$ has the form $M^2_{j}=\mathbb V^2_\pm\times\{0\}\cup_J\mathbb V^2_\pm\times\{1\}$, where $J\colon \partial\mathbb V^2_\pm\times\{0\}\to\partial\mathbb V^2_\pm\times\{1\}$ is a homeomorphism defined as $J(s,0)=(j(s),1)$ for some homeomorphism $j\colon \partial \mathbb V^2_\pm\to \partial \mathbb V^2_\pm$. 

If the periodic orbit is untwisted then its tubular neighbourhood is an annulus and its boundary has two connected components each of them is homeomorphic to the circle. If the periodic orbit is  twisted then its tubular neighbourhood  is a M\"obius band and its boundary is homeomorphic to the circle. Let $\mathbb S^1=\{e^{i\varphi},\,\varphi\in\mathbb R\}$, $\mathbb S^0=\{-1,+1\}$. Define the following  diffeomorphisms on the manifolds  $\partial\mathbb V^2_+\cong\mathbb S^1\times\mathbb S^0$:

1. $j_1(\varphi,\pm 1)=(\varphi,\pm 1)$;

2. $j_2(\varphi,\pm 1)=(-\varphi,\pm 1)$;

3. $j_3(\varphi,-1)=(-\varphi,-1),\,j_3(\varphi,+1)=(\varphi,+1)$.

Define the following diffeomorphisms on the manifold $\partial\mathbb V^2_-\cong\mathbb S^1$:

4. $j_4(\varphi)=\varphi$;

5. $j_5(\varphi)=-\varphi$.

Pictures \ref{pictorus_2}, \ref{phaseG2M2} provide the phase portraits of the model flows corresponding to the given maps. The sign ``$+$'' means the gluing with the map $\varphi$ and the sign ``$-$'' with the map $-\varphi$.
\begin{figure}[h!]
	\begin{center}
		\begin{minipage}[h]{0.32\linewidth}
\center{\includegraphics[width=1\linewidth]{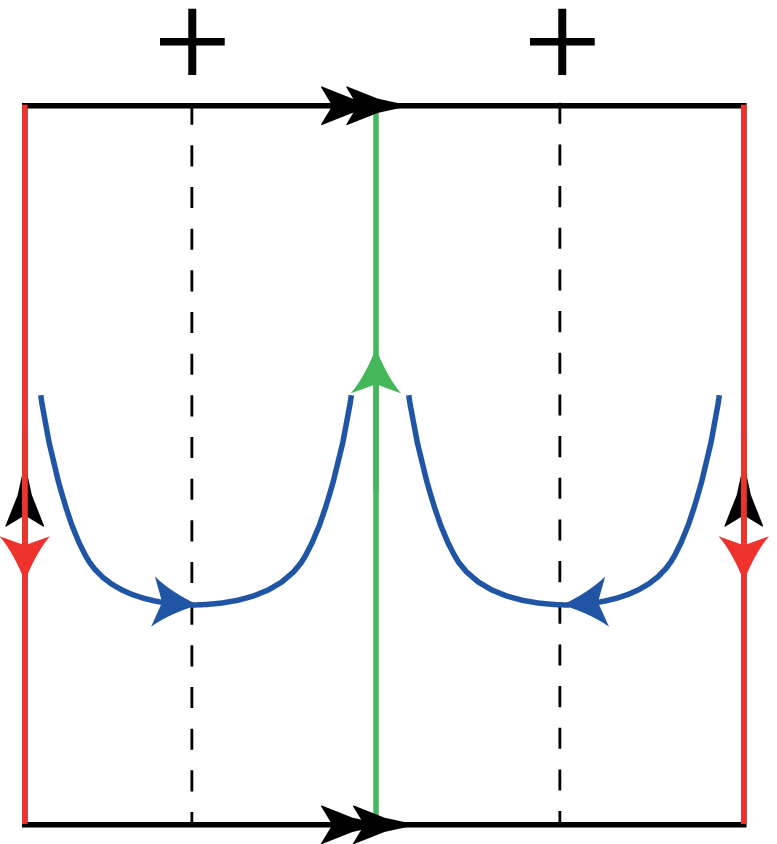} $f^t_{j_1}$}
		\end{minipage}
		\begin{minipage}[h]{0.32\linewidth}
\center{\includegraphics[width=1\linewidth]{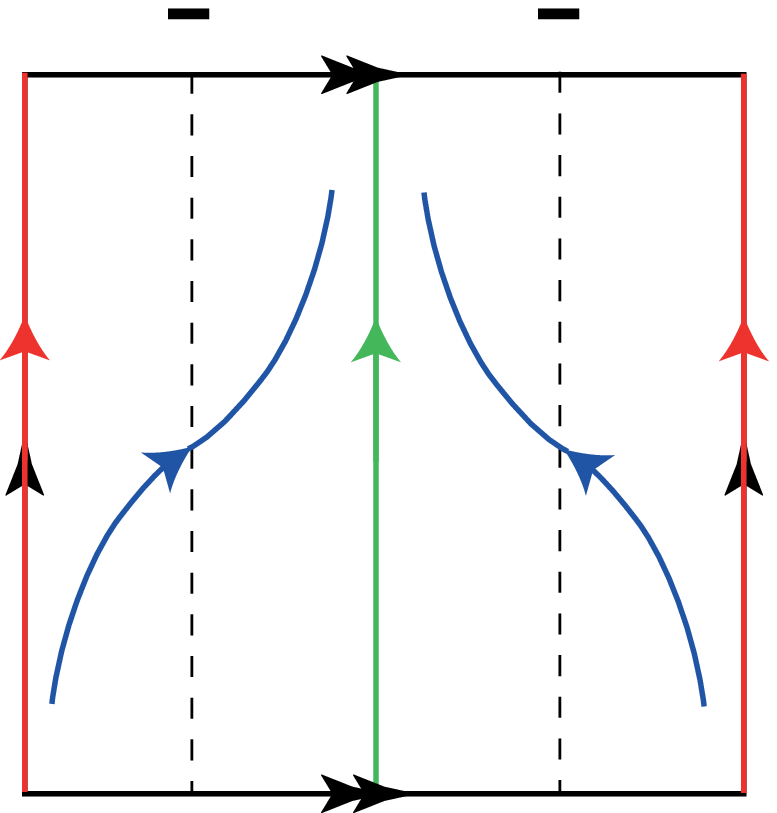} $f^t_{j_2}$}
		\end{minipage}
	\end{center}
	\caption{Flows on the torus}\label{pictorus_2}
	\begin{minipage}[h]{0.32\linewidth}
\center{\includegraphics[width=1\linewidth]{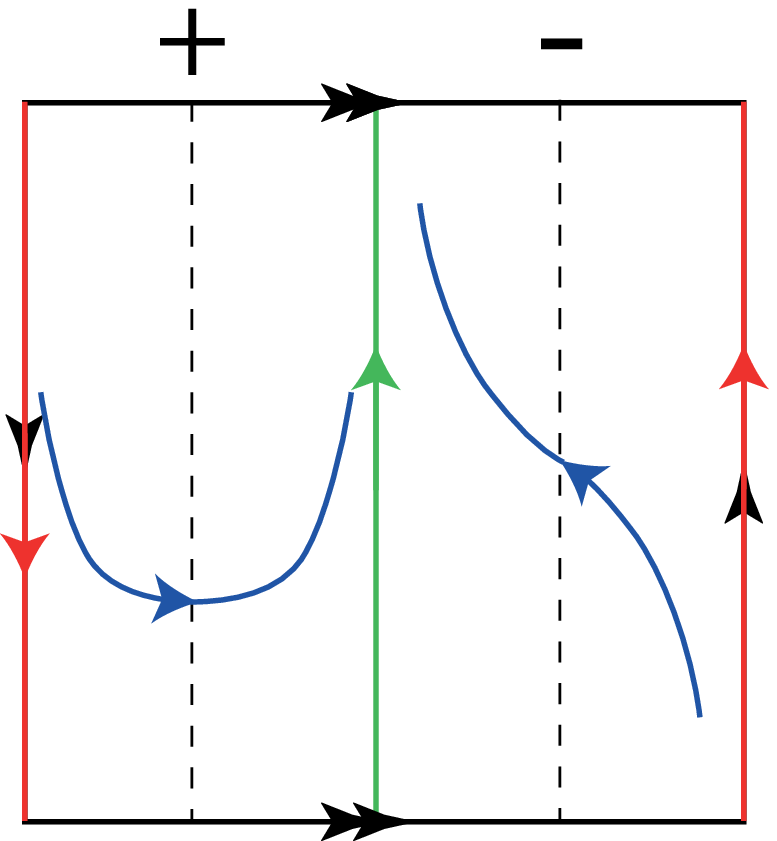} $f^t_{j_3}$}
	\end{minipage}
		\begin{minipage}[h]{0.32\linewidth}
\center{\includegraphics[width=1\linewidth]{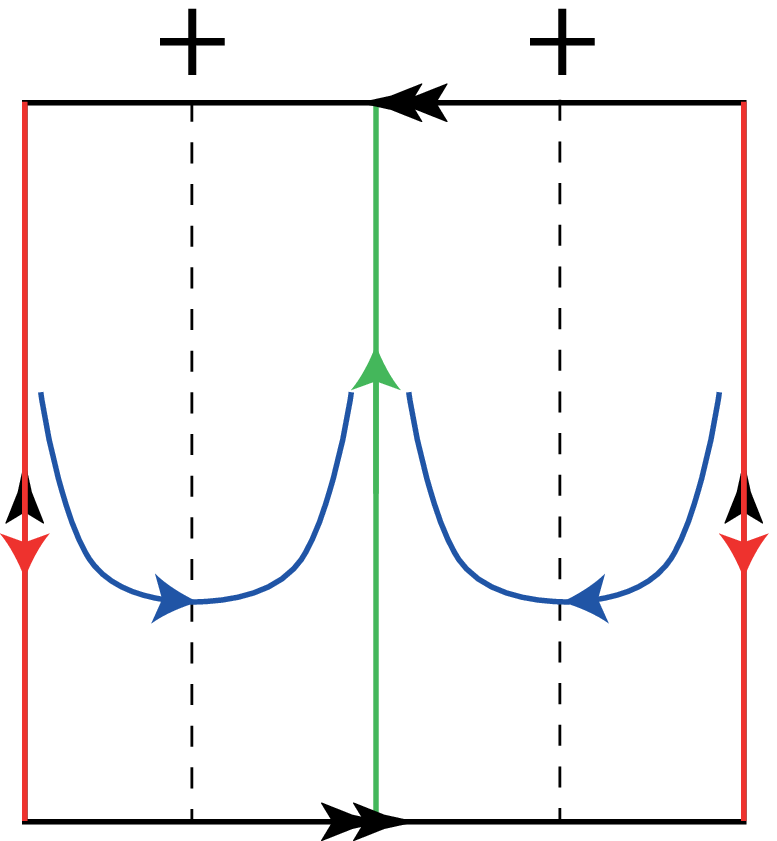} $f^t_{j_4}$}
	\end{minipage}
	\begin{minipage}[h]{0.32\linewidth}
\center{\includegraphics[width=1\linewidth]{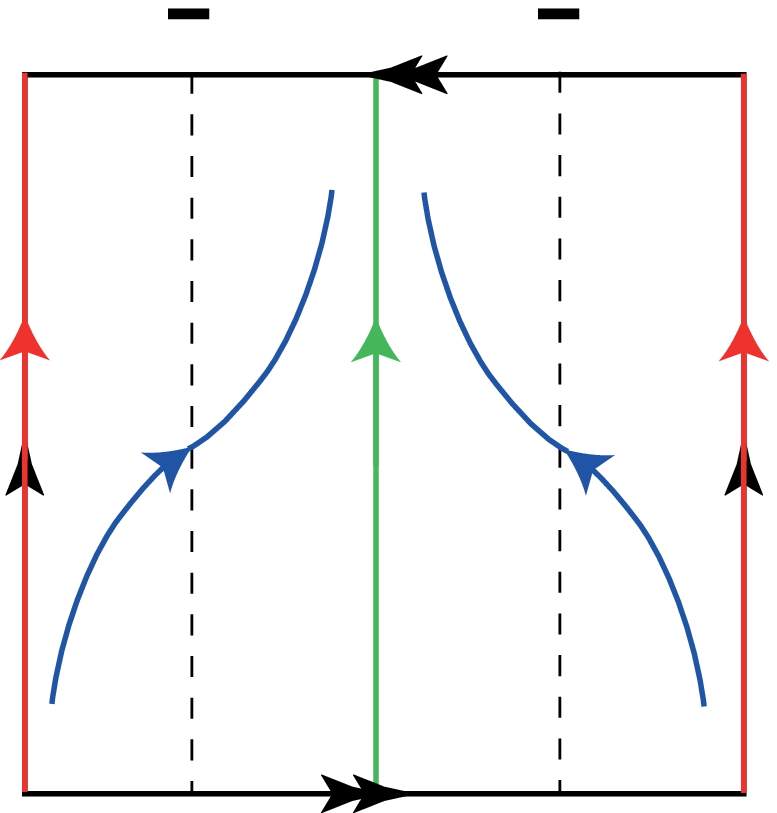} $f^t_{j_5}$}
	\end{minipage}
	\caption{Flows on the Klein bottle}\label{phaseG2M2}
\end{figure}

\begin{lemma}\label{surf_not_eqv} $ $
\begin{enumerate}
\item Every model flow $f^t_{j}:M^2_+\to M^2_+$ is topologically equivalent either to  $f^t_{j_1}$ or to $f^t_{j_2}$, herewith $f^t_{j_1}$ and $f^t_{j_2}$  are non topologically
equivalent to each other.
\item Every model flow $f^t_{j}:M^2_-\to M^2_-$ is topologically equivalent either to $f^t_{j_3}$ or to $f^t_{j_4}$, or to $f^t_{j_5}$, herewith $f^t_{j_3}$, $f^t_{j_4}$ and  $f^t_{j_5}$ are pairwise non topologically equivalent.
\end{enumerate}
\end{lemma}
\begin{proof} Since the fundamental group of the circle is isomorphic to the group $\mathbb Z$ then every orientation-preserving homeomorphism induces identical action in the fundamental group and the orientation-reversing induces action which changes the  homotopy type of the curve on the opposite. Then, by virtue of  Theorem~\ref{criteria}, flows $f^t_j\colon M^2_j\to M^2_j,\,f^t_{j'}\colon M^2_{j'}\to M^2_{j'}$ are topologically equivalent if and only if there is the orientation-preserving homeomorphism $h_0\colon  \partial \mathbb V^2_\pm \to \partial \mathbb V^2_\pm$ such that the homeomorphism $h_1 = j' h_0j^{-1}$ preserves orientation. That is equivalent to the fact that $j'j^{-1}$ preserves orientation.
	
The exhaustive search of all  possible combinations of the  orientability of the homeomorphism $j$ on the connected components gives that the homeomorphism $j_ij^{-1}$ preserves the orientation exactly for unique  value $i\in\{1,\dots,5\}$. Moreover, if $i=1,2$ the ambient manifold is the torus, but for $i=3,4,5$ is the Klein bottle. Finally, if $i=4,5$ then both orbits are twisted and untwisted in the remaining cases.
\end{proof}

\section{Classification of 3-dimensional model flows}

Let $f^t_j\colon M^3_{j}\to M^3_{j}$ be a three-dimensional model flow. Then the ambient manifold $M^3_{j}$ has a form $M^3_{j}=\mathbb V^3_\pm\times\{0\}\cup_J\mathbb V^3_\pm\times\{1\}$ where $J\colon \partial\mathbb V^3_\pm\times\{0\}\to\partial\mathbb V^3_\pm\times\{1\}$ is a homeomorphism defined by  $J(s,0)=(j(s),1)$ for  $j\colon \partial \mathbb V^3_\pm\to \partial \mathbb V^3_\pm$. It is easy to see that the manifold $\mathbb V^3_+$ is the solid torus whereas $\mathbb V^3_-$ is the solid Klein bottle. In the first case ambient manifold $M^3_{j}$ is a lens space, which is orientable and, by Proposition~\ref{3lens_class} below, there are countably many pairwise non-homeomorphic lens spaces. All the manifolds obtained by gluing the solid Klein bottles, on the contrary, are homeomorphic to the $S^2\widetilde{\times} S^1$ (see e.g~\cite[section 3.1(c)]{3manif2bundl}). 

Let us prove Theorem~\ref{G2M3} separately for orientable and non-orientable manifolds.

\subsection{Orientable case}
On the torus $\partial{\mathbb V}^3_+$ the curves $\alpha_+$ and $\beta_+$ are generators in the torus fundamental group. The oriented curves $\alpha_+, \beta_+$ on the torus $\partial{\mathbb V}^3_+$ are said to be the \emph{longitude and meridian} respectively. The action of the homeomorphism $j$ in the fundamental group of the torus is uniquely defined by an unimodular integer matrix $$j_*=\begin{pmatrix}
	r & p\\
	s & q
\end{pmatrix}.$$
\begin{figure}[h]
	\centerline{\includegraphics[width=0.9\textwidth]{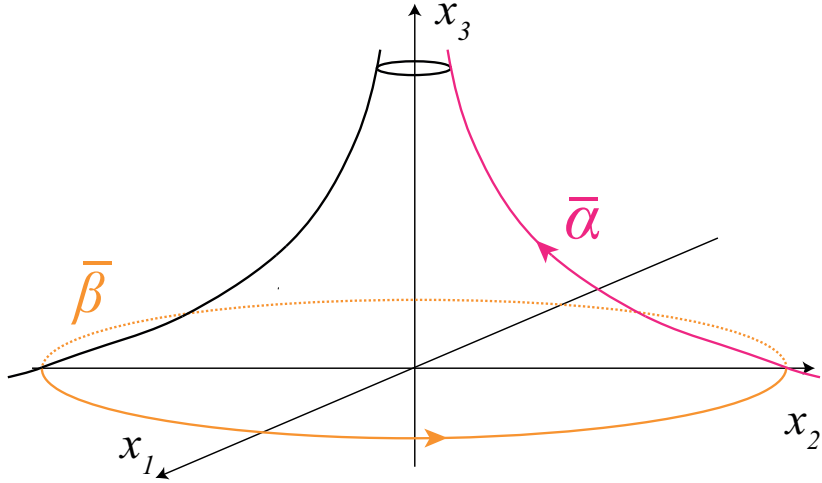}}
	\caption{Longitude and meridian pre-images on $\bar V^3$}\label{V_a_b_}
\end{figure}

Recall that the presentation of the lens space in the form $M^3_{j}=\mathbb V^3_\pm\times\{0\}\cup_J\mathbb V^3_\pm\times\{1\}$ is called a {\it Heegaard decomposition} and $\Sigma_j$ is called a {\it Heegaard torus}. Using the uniqueness up to isotopy of the Heegaard torus in every lens space  (see, for example, \cite{Hat}), we will suppose below  that every homeomorphism $h:M^3_{j}\to M^3_{j'}$ possesses a property $$\eta_1(h(p_j(\mathbb V^3_+\times\{0\})))=p_{j'}(\mathbb V^3_+\times\{0\}),$$ here $\eta_t:M^3_{j'}\to M^3_{j'},\,t\in[0,1]$ is an isotopy which moves $h(\Sigma_j)$ to $\Sigma_{j'}$. The classification of the lens spaces up to a such sort of homeomorphisms has the following view. 

\begin{proposition}[Lens space classification, \cite{BO}]\label{len-class} Two lens spaces $M^3_{j}$ and $M^3_{j'}$ are homeomorphic if and only if  the induced isomorphisms $j_*=\begin{pmatrix}
		r & p\\
		s & q
	\end{pmatrix},\, j'_*=\begin{pmatrix}
		r' & p'\\
		s' & q'
	\end{pmatrix}$ satisfy the conditions $|p'|= |p|$ and $q' \equiv \pm q \pmod{|p|}$\footnote{Considering the fact  that matrices $j_*,\,j'_*$ are unimodular it is easy to establish one more criteria for two lens spaces to be homeomorphic. Namely, two lens spaces $M^3_{j}$ and $M^3_{j'}$ are homeomorphic if and only if  the ifinduced isomorphisms $j_*=\begin{pmatrix}
		r & p\\
		s & q
	\end{pmatrix},\, j'_*=\begin{pmatrix}
		r' & p'\\
		s' & q'
	\end{pmatrix}$ satisfy the relations $|p'|= |p|$ and $r' \equiv \pm r \pmod{|p|}$.}.\label{3lens_class}
\end{proposition}

The following lemma is a refinement of Proposition \ref{len-class} which we need to prove Theorem~\ref{3lens_class}.

\begin{lemma}\label{p=p} A homeomorphism $h_0\colon\partial\mathbb V^3_+\to\partial\mathbb V^3_+$ with property  $i_*h_{k*} = i_*$ for $k=0,1$  and $h_1 = j' h_0 j^{-1}$ there is if and only if  the induced isomorphisms $j_*=\begin{pmatrix}
		r & p\\
		s & q
	\end{pmatrix},\, j'_*=\begin{pmatrix}
		r' & p'\\
		s' & q'
	\end{pmatrix}$ satisfy the conditions $|p'| = |p|, r' \equiv r \pmod{|p|}$.
\end{lemma}

\begin{proof} \textit{Necessity.} Assume there is a homeomorphism $h_0\colon \partial \mathbb V^3_+\to \partial \mathbb V^3_+$ such that for $h_0$ and $h_1 = j' h_0 j^{-1}$ the condition $i_*h_{k*} = i_*, k=0,1$ holds. Then the homeomorphism $h_k$ acts in the fundamental group of the torus with the matrix
	$$h_{k *} = \begin{pmatrix}
		1 & 0\\
		m_k & \pm 1
	\end{pmatrix}$$
	where $m_k$ is integer and
	$$h_{1*} j_* = j'_* h_{0*}.$$
	The last relation can be written in the following matrix form:
	\begin{equation}\label{matr:p=p}
		\begin{pmatrix}
			1 & 0 \\
			m_1& \pm 1
		\end{pmatrix}  
		\begin{pmatrix}
			r & p \\
			s & q
		\end{pmatrix}
		=
		\begin{pmatrix}
			r' & p' \\
			s' & q'
		\end{pmatrix} \begin{pmatrix}
			1 & 0\\
			m_0 & \pm 1
		\end{pmatrix}. 
	\end{equation}
	So, we get the equalities $p = \pm p',\ r = r' + m_0p'$, which are equivalent to $|p'|= |p|$ and $r' \equiv r \pmod{|p|}$. 
	
\textit{Sufficiency.} Let the elements of matrices $j_*=\begin{pmatrix}
		r & p\\
		s & q
	\end{pmatrix},\, j'_*=\begin{pmatrix}
		r' & p'\\
		s' & q'
	\end{pmatrix}$ satisfy the relations $|p'|= |p|$ and $r' \equiv r \pmod{|p|}$. Thus 
	\begin{equation}\label{pr0}
		p = \delta_0 p',\,\,r = r' + m_0p'
	\end{equation} for some $\delta_0\in\{-1,1\},\ m_0\in\mathbb{Z}$.
	Let $h_0\colon \partial \mathbb{V}^3_+\to\partial \mathbb{V}^3_+$ be the algebraic torus automorphism defined by the matrix $h_{0*} = \begin{pmatrix}
		1 & 0\\
		m_0 & \delta_0
	\end{pmatrix}$. 
	Then $i_*h_{0*} = i_*$. Formula (\ref{matr:p=p}) and the fact that all the matrices are unimodular give us that the homeomorphism $h_1=j'h_0j^{-1}\colon\partial \mathbb{V}^3_+\to\partial \mathbb{V}^3_+$ induces the isomorphism, defined by the matrix $h_{1*} = \begin{pmatrix}
		1 & 0\\
		m_1 & \delta_1
	\end{pmatrix}$ for some $\delta_1\in\{-1,1\},\ m_1\in\mathbb{Z}$. 
	Thus $i_*h_{1*} = i_*$.
\end{proof}

\begin{lemma}\label{3+}
	Up to topological equivalence there is only one flow on both $\mathbb S^3$ and $\mathbb RP^3$ and two flows on the remaining lens spaces.
\end{lemma}
\begin{proof} By virtue of  Proposition \ref{len-class} and Lemma~\ref{p=p},  on the same lens space there are either one or two topological equivalence classes of flows from $G_2(M^3_+)$ and the cases are distinguished by the following condition: whether for two coprime numbers $p\geqslant 0,\,r\in\mathbb Z_p$ there are two numbers $n_1,\,n_2$ such that \begin{equation}\label{-+} r + n_1p = - r + n_2p
	\end{equation}
Check whether condition (\ref{-+}) holds for all the possible values of $p$.
\begin{enumerate}
		\item if $p=0$ than condition (\ref{-+}) does not hold, since the equality is true only if $r=0$ but in this case $r,p$ are not coprime;
\item if $p=1$ the condition holds for $r=0$, it means that up to topological equivalence there is a  unique flow in $G_2(\mathbb{S}^3)$;
\item if $p=2k$ then $r = k(n_2 - n_1)$. Considering the fact that $(r, p) = 1$ deduce that $k = 1, r = 1$, which means that up to topological equivalence there is a unique flow in $G_2(\mathbb{R}P^3)$;
\item for $p= 2k + 1, k > 0$ condition (\ref{-+}) never holds since $n_2 - n_1$ is even and $(r, p) \neq 1$.
	\end{enumerate}
\end{proof}

\subsection{Non-orientable case}
By  \cite{Lickorish} every homeomorphism $j\colon \partial \mathbb V^3_-\to \partial \mathbb V^3_-$ satisfy either $i_*(j_*([c]))=i_*([c])$ or $i_*(j_*([c]))=-i_*([c])$. Then Theorem~\ref{criteria} implies that there are two topological equivalence classes of the flows in $G_2(\mathbb S^2\widetilde{\times} S^1)$.

\section{Classification of $n$-dimensional model flows for $n>3$}
According to~\cite{Max}, every homeomophism $j\colon \partial \mathbb V^n_+\to \partial \mathbb V^n_+$ can be extended to a  homeomorphism $\phi\colon  \mathbb V^n_+\to  \mathbb V^n_+ $. Thus  the only manifold obtained by gluing two copies of $\mathbb V^n_+$ along the boundaries is $\mathbb S^{n-1}\times \mathbb{S}^1$. Similarly, the results of the article~\cite{Kwasik} imply that the only manifold obtained by gluing two copies of $\mathbb V^n_-$ along the boundaries is $\mathbb S^{n-1} \tilde\times \mathbb S^1$.

Since the fundamental group $\pi_1(\partial \mathbb V^n_\pm)$ is isomorphic to the group $\mathbb Z$ then any homeomorphism $j\colon \partial \mathbb V^n_\pm\to \partial \mathbb V^n_\pm$ satisfies either $i_*(j_*([c]))=i_*([c])$ or $i_*(j_*([c]))=-i_*([c])$. Thus, Theorem~\ref{criteria} implies that each of the manifolds $\mathbb S^{n-1}\times \mathbb{S}^1$ and $\mathbb S^{n-1} \tilde\times \mathbb S^1$ admits two topological equivalence classes of the flows from $G_2(M^n), n>3$.

\end{document}